\documentclass{article}
\usepackage{amsmath,amsfonts,graphicx,amsthm,float}
\newcommand{\half}{\ensuremath{\protect\tfrac{1}{2}}}
\newcommand{\quarter}{\ensuremath{\protect\tfrac{1}{4}}}
\newcommand{\third}{\ensuremath{\protect\tfrac{1}{3}}}
\newcommand{\FLOOR}[1]{\ensuremath{\protect\left\lfloor#1\right\rfloor}}
\theoremstyle{plain}
\newtheorem{theorem}{Theorem}
\newtheorem{lemma}{Lemma}

\begin{document}

\title{\vspace*{-4ex}\textbf{The chromatic number of the\\ convex segment
    disjointness graph}\footnote{Presented at the XIV Spanish Meeting
    on Computational Geometry Alcal\'a de Henares, Spain, June 27--30,
    2011}\\[1ex]{\it\large Dedicat al nostre amic i mestre Ferran Hurtado}
}

\author{Ruy Fabila-Monroy\footnote{Departamento de Matem\'aticas,
    Centro de Investigaci\'on y Estudios Avanzados del Instituto
    Polit\'ecnico Nacional, M\'exico, D.F., M\'exico
    (\texttt{ruyfabila@math.cinvestav.edu.mx}). Supported by an
    Endeavour Fellowship from the Department of Education, Employment
    and Workplace Relations of the Australian Government.}  \and David
  R. Wood\footnote{Department of Mathematics and Statistics, The
    Univesity of Melbourne, Melbourne, Australia
    (\texttt{woodd@unimelb.edu.au}). Supported by a QEII Fellowship
    from the Australian Research Council.}}

\maketitle

\begin{abstract}
  Let $P$ be a set of $n$ points in general and convex position in the
  plane.  Let $D_n$ be the graph whose vertex set is the set of all
  line segments with endpoints in $P$, where disjoint segments are
  adjacent.  The chromatic number of this graph was first studied by
  Araujo~et~al.~[\emph{CGTA}, 2005]. The previous best bounds are
  $\frac{3n}{4}\leq \chi(D_n) <n-\sqrt{\frac{n}{2}}$ (ignoring lower
  order terms). In this paper we improve the lower bound to
  $\chi(D_n)\geq n-\sqrt{2n}$, to conclude a near-tight bound on
  $\chi(D_n)$.
\end{abstract}

\section{Introduction}

Throughout this paper, $P$ is a set of $n>3$ points in general and
convex position in the plane.  The \emph{convex segment disjointness
  graph}, denoted by $D_n$, is the graph whose vertex set is the set
of all line segments with endpoints in $P$, where two vertices are
adjacent if the corresponding segments are disjoint. Obviously $D_n$
does not depend on the choice of $P$.  This graph and other related
graphs, were introduced by Araujo, Dumitrescu, Hurtado, Noy and
Urrutia~\cite{ADHNU-CGTA}, who proved the following bounds on the
chromatic number of $D_n$:
\begin{equation*}
  \label{eqn:ADHNU}
  2\FLOOR{\third(n+1)}-1\leq \chi(D_n) <n-\half\FLOOR{\log n}\enspace.
\end{equation*}
Both bounds were improved by Dujmovi\'c and Wood~\cite{Antithickness}
to
\begin{equation*}
  \tfrac{3}{4}(n-2)\leq \chi(D_n) 
  <n-\sqrt{\half n}-\half(\ln n)+4\enspace.
\end{equation*}
In this paper we improve the lower bound to conclude near-tight bounds
on $\chi(D_n)$.
\begin{theorem}\label{thm:main}
$$n-\sqrt{2n+\quarter}+\half \leq \chi(D_n)
<n-\sqrt{\half n}-\half(\ln n)+4\enspace.$$
\end{theorem}

The proof of Theorem~\ref{thm:main} is based on the observation that
eah colour class in a colouring of $D_n$ is a convex thrackle. We then
prove that two maximal convex thrackles must share an edge in
common. From this we prove a tight upper bound on the number of edges
in the union of $k$ maximal convex thrackles. Theorem~\ref{thm:main}
quickly follows.

\section{Convex thrackles}

A \emph{convex thrackle} on $P$ is a geometric graph with vertex set
$P$ such that every pair of edges intersect; that is, they have a
common endpoint or they cross. Observe that a geometric graph $H$ on
$P$ is a convex thrackle if and only if $E(H)$ forms an independent
set in $D_n$. A convex thrackle is \emph{maximal} if it is
edge-maximal.  As illustrated in Figure~\ref{fig:thrackle}(a), it is
well known and easily proved that every maximal convex thrackle $T$
consists of an odd cycle $C(T)$ together with some degree $1$ vertices
adjacent to vertices of $C(T)$; see
\cite{CN-DCG00,CN09,CN10,Antithickness,FS35,HopfPann34,LPS-DCG97,Woodall-Thrackles}. In
particular, $T$ has $n$ edges. For each vertex $v$ in $C(T)$, let
$W_T(v)$ be the convex wedge with apex $v$, such that the boundary
rays of $W_T(v)$ contain the neighbours of $v$ in $C(T)$. Every
degree-1 vertex $u$ of $T$ lies in a unique wedge and the apex of this
wedge is the only neighbour of $u$ in $T$.

\begin{figure}[H]
  \begin{center}
    \includegraphics[width=\textwidth]{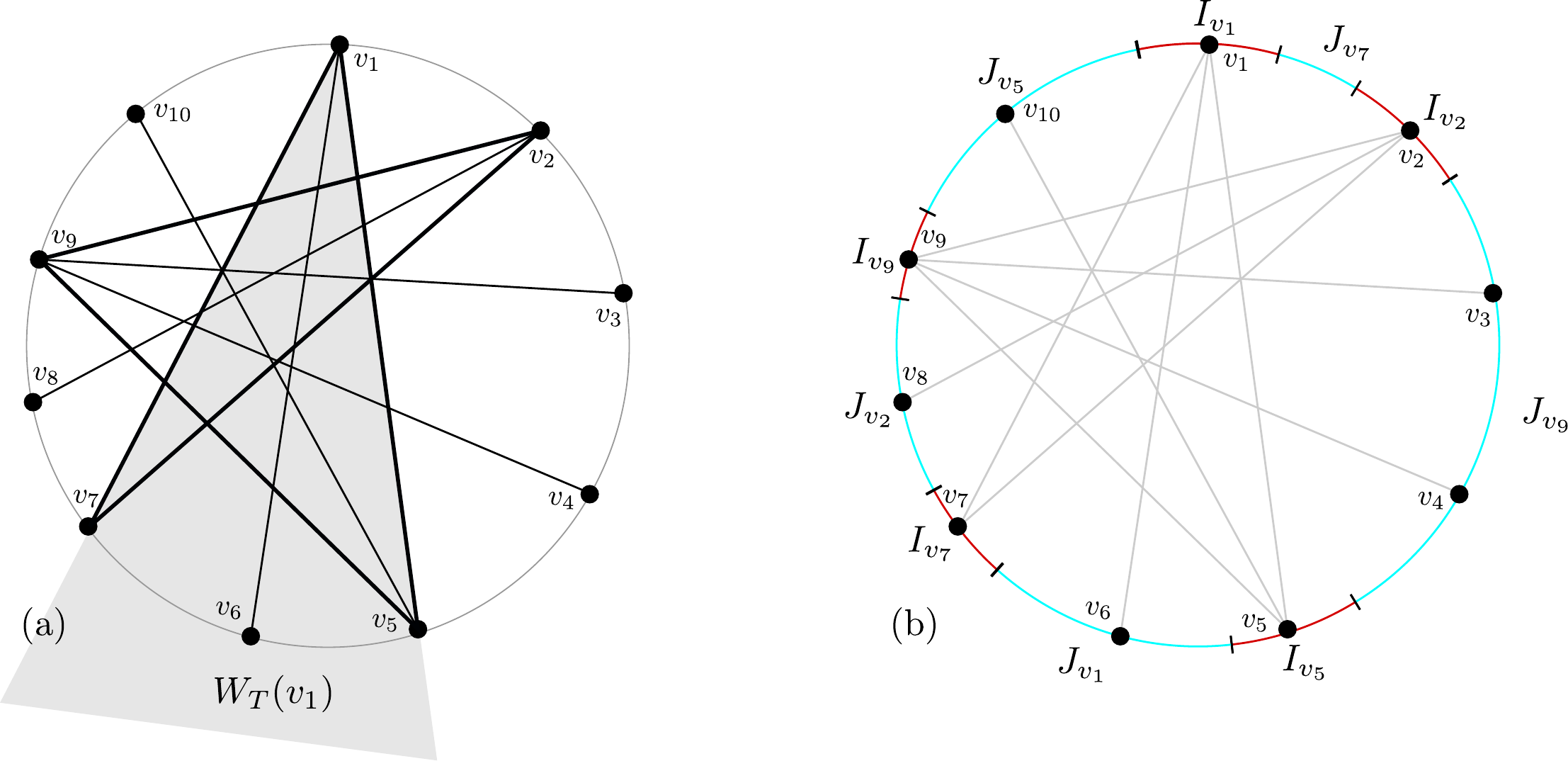}
  \end{center}
  \caption{(a) maximal convex thrackle, (b) the intervals pairs
    $(I_u,J_u)$} \label{fig:thrackle}
\end{figure}

\section{Convex thrackles and free $\mathbb{Z}_2$-actions of $S^1$}

A \emph{$\mathbb{Z}_2$-action} on the unit circle $S^1$ is a
homeomorphism $f:S^1\rightarrow S^1$ such that $f(f(x))=x$ for all
$x\in S^1$. We say that $f$ is \emph{free} if $f(x)\neq x$ for all
$x\in S^1$.
 
\begin{lemma}\label{lem:agreepoint}
  If $f$ and $g$ are free $\mathbb{Z}_2$-actions of $S^1$, then
  $f(x)=g(x)$ for some point $x \in S^1$.
\end{lemma}

\begin{proof}
  For points $x,y\in S^1$, let $\overrightarrow{xy}$ be the clockwise
  arc from $x$ to $y$ in $S^1$.  Let $x_0\in S^1$. If $f(x_0)=g(x_0)$
  then we are done. Now assume that $f(x_0)\neq g(x_0)$. Without loss
  of generality, $x_0,g(x_0),f(x_0)$ appear in this clockwise order
  around $S^1$.  Paramaterise $\overrightarrow{x_0g(x_0)}$ with a
  continuous injective function $p:[0,1]\rightarrow
  \overrightarrow{x_0g(x_0)}$, such that $p(0)=x_0$ and
  $p(1)=g(x_0)$. Assume that $g(p(t))\neq f(p(t))$ for all
  $t\in[0,1]$, otherwise we are done. Since $g$ is free, $p(t)\neq
  g(p(t))$ for all $t\in[0,1]$. Thus
  $g(p([0,1]))=\overrightarrow{g(p(0))g(p(1))}=\overrightarrow{g(x_0)x_0}$.
  Also $f(p([0,1]))=\overrightarrow{f(x_0)f(p(1))}$, as otherwise
  $g(p(t))= f(p(t))$ for some $t\in[0,1]$.  This implies that
  $p(t),g(p(t)),f(p(t))$ appear in this clockwise order around
  $S^1$. In particular, with $t=1$, we have $f(p(1))\in
  \overrightarrow{x_0g(x_0)}$. Thus
  $x_0\in\overrightarrow{f(x_0)f(p(1))}$.  Hence $x_0=f(p(t))$ for
  some $t\in[0,1]$.  Since $f$ is a $\mathbb{Z}_2$-action,
  $f(x_0)=p(t)$. This is a contradiction since
  $p(t)\in\overrightarrow{x_0g(x_0)}$ but
  $f(x_0)\not\in\overrightarrow{x_0g(x_0)}$.
\end{proof}

Assume that $P$ lies on $S^1$.  Let $T$ be a maximal convex thrackle
on $P$.  As illustrated in Figure~\ref{fig:thrackle}(b), for each
vertex $u$ in $C(T)$, let $(I_u,J_u)$ be a pair of closed intervals of
$S^1$ defined as follows. Interval $I_u$ contains $u$ and bounded by
the points of $S^1$ that are $1/3$ of the way towards the first points
of $P$ in the clockwise and anticlockwise direction from $u$.  Let $v$
and $w$ be the neighbours of $u$ in $C(T)$, so that $v$ is before $w$
in the clockwise direction from $u$. Let $p$ be the endpoint of $I_v$
in the clockwise direction from $v$. Let $q$ be the endpoint of $I_w$
in the anticlockwise direction from $w$. Then $J_u$ is the interval
bounded by $p$ and $q$ and not containing $u$.  Define $f_T:S^1
\longrightarrow S^1$ as follows. For each $v\in C(T)$, map the
anticlockwise endpoint of $I_v$ to the anticlockwise endpoint of
$J_v$, map the clockwise endpoint of $I_v$ to the clockwise endpoint
of $J_v$, and extend $f_T$ linearly for the interior points of $I_v$
and $J_v$, such that $f_T(I_v)=J_v$ and $f_T(J_v)=I_v$. Since the
intervals $I_v$ and $J_v$ are disjoint, $f_T$ is a free
$\mathbb{Z}_2$-action of $S^1$.

\begin{lemma}\label{lem:common_edge}
  Let $T_1$ and $T_2$ be maximal convex thrackles on $P$, such that
  $C(T_1)\cap C(T_2)=\emptyset$.  Then there is an edge in $T_1\cap
  T_2$, with one endpoint in $C(T_1)$ and one endpoint in $C(T_2)$.
\end{lemma}

\begin{proof}[Topological proof]
  By Lemma~\ref{lem:agreepoint}, there exists $x \in S^1$ such that
  $f_{T_1}(x)=y=f_{T_2}(x)$.  Let $u\in C(T_1)$ and $v \in C(T_2)$ so
  that $x \in I_u \cup J_u$ and $x \in I_v \cup J_v$, where
  $(I_u,J_u)$ and $(I_v,J_v)$ are defined with respect to $T_1$ and
  $T_2$ respectively. Since $C(T_1)\cap C(T_2)=\emptyset$, we have
  $u\neq v$ and $I_u\cap I_v=\emptyset$. Thus $x\not\in I_u\cap
  I_v$. If $x\in J_u\cap J_v$ then $y\in I_u\cap I_v$, implying
  $u=v$. Thus $x\not\in J_u\cap J_v$.  Hence $x\in(I_u\cap
  J_v)\cup(J_u\cap I_v)$.  Without loss of generality, $x\in I_u\cap
  J_v$. Thus $y\in J_u\cap I_v$. If $I_u\cap J_v=\{x\}$ then $x$ is an
  endpoint of both $I_u$ and $J_v$, implying $u\in C(T_2)$, which is a
  contradiction. Thus $I_u\cap J_v$ contains points other than $x$. It
  follows that $I_u\subset J_v$ and $I_v\subset J_u$.  Therefore the
  edge $uv$ is in both $T_1$ and $T_2$. Moreover one endpoint of $uv$
  is in $C(T_1)$ and one endpoint is in $C(T_2)$.
\end{proof}

\begin{proof}[Combinatorial Proof] 
  Let $H$ be the directed multigraph with vertex set $C(T_1)\cup
  C(T_2)$, where there is a \emph{blue} arc $uv$ in $H$ if $u$ is in
  $W_{T_1}(v)$ and there is a \emph{red} arc $uv$ in $H$ if $u$ is in
  $W_{T_2}(v)$. Since $C(T_1)\cap C(T_2)=\emptyset$, every vertex of
  $H$ has outdegree $1$. Therefore $|E(H)|=|V(H)|$ and there is a
  cycle $\Gamma$ in the undirected multigraph underlying $H$.  In
  fact, since every vertex has outdegree 1, $\Gamma$ is a directed
  cycle. By construction, vertices in $H$ are not incident to an
  incoming and an outgoing edge of the same color. Thus $\Gamma$
  alternates between blue and red arcs. The red edges of $\Gamma$ form
  a matching as well as the blue edges, both of which are
  thrackles. However, there is only one matching thrackle on a set of
  points in convex position.  Therefore $\Gamma$ is a 2-cycle and the
  result follows.
\end{proof}

\section{Main Results} \label{sec:lower_bound}

\begin{theorem} \label{thm:overcount} For every set $P$ of $n$ points
  in convex and general position, the union of $k$ maximal convex
  thrackles on $P$ has at most $k n-\binom{k}{2}$ edges.
\end{theorem}

\begin{proof}
  For a set $\mathcal{T}$ of $k$ maximal convex thrackles on $P$,
  define
$$r(\mathcal{T}):=\left
  |\{(v,T_i,T_j): v \in C(T_i) \cap C(T_j), T_i,T_j \in \mathcal{T}
  \textrm{ and } T_i \neq T_j\} \right |\enspace.$$ The proof proceeds
by induction on $r(\mathcal{T})$.

Suppose that $r(\mathcal{T})=0$.  Thus $C(T_i)\cap C(T_j)=\emptyset$
for all distinct $T_i,T_j\in \mathcal{T}$.  By
Lemma~\ref{lem:common_edge}, $T_i$ and $T_j$ have an edge in common,
with one endpoint in $C(T_i)$ and one endpoint in $C(T_j)$.  Hence
distinct pairs of thrackles have distinct edges in common.  Since
every maximal convex thrackle has $n$ edges and we overcount at least
one edge for every pair, the total number of edges is at most
$kn-\binom{k}{2}$.

Now assume that $r(\mathcal{T}) >0$. Thus there is a vertex $v$ and a
pair of thrackles $T_i$ and $T_j$, such that $v \in C(T_i) \cap
C(T_j)$.  As illustrated in Figure~\ref{fig:split}, replace $v$ by two
consecutive vertices $v'$ and $v''$ on $P$, where $v'$ replaces $v$ in
every thrackle except $T_j$, and $v''$ replaces $v$ in $T_j$. Add one
edge to each thrackle so that it is maximal.  Let $\mathcal{T}'$ be
the resulting set of thrackles.  Observe that
$r(\mathcal{T'})=r(\mathcal{T})-1$, and the number of edges in
$\mathcal{T'}$ equals the number of edges in $\mathcal{T}$ plus
$k$. By induction, $\mathcal{T}'$ has at most $k (n+1)-\binom{k}{2}$
edges, implying $\mathcal{T}$ has at most $k n -\binom{k}{2}$ edges.
\end{proof}

\begin{figure}
  \begin{center}
    \includegraphics[scale=0.7]{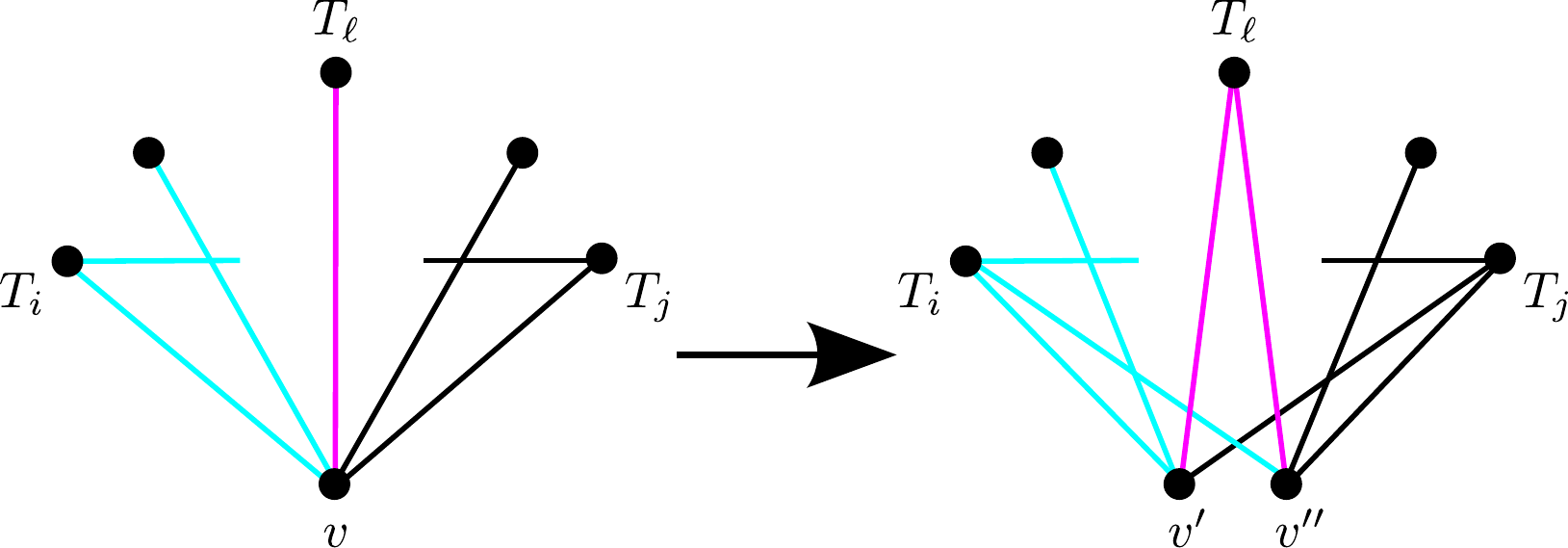}
  \end{center}
  \caption{\label{fig:split} Construction in the proof of
    Theorem~\ref{thm:overcount}.}
\end{figure}

We now show that Theorem~\ref{thm:overcount} is best possible for all
$n\geq 2k$. Let $S$ be a set of $k$ vertices in $P$ with no two
consecutive vertices in $S$. If $v\in S$ and $x,v,y$ are consecutive
in this order in $P$, then $T_v:=\{vw:w\in
P\setminus\{v\})\}\cup\{xy\}$ is a maximal convex thrackle, and
$\{T_v:v\in S\}$ has exactly $k n-\binom{k}{2}$ edges in total.

\begin{proof}[Proof of  Theorem~\ref{thm:main}]
  If $\chi(D_n)=k$ then, there are $k$ convex thrackles whose union is
  the complete geometric graph on $P$.  Possibly add edges to obtain
  $k$ maximal convex thrackles with $\binom{n}{2}$ edges in total.  By
  Theorem~\ref{thm:overcount}, $\binom{n}{2}\leq kn-\binom{k}{2}$.
  The quadratic formula implies the result.
\end{proof}

\def\cprime{$'$}
\def\soft#1{\leavevmode\setbox0=\hbox{h}\dimen7=\ht0\advance \dimen7
  by-1ex\relax\if t#1\relax\rlap{\raise.6\dimen7
    \hbox{\kern.3ex\char'47}}#1\relax\else\if T#1\relax
  \rlap{\raise.5\dimen7\hbox{\kern1.3ex\char'47}}#1\relax \else\if
  d#1\relax\rlap{\raise.5\dimen7\hbox{\kern.9ex
      \char'47}}#1\relax\else\if D#1\relax\rlap{\raise.5\dimen7
    \hbox{\kern1.4ex\char'47}}#1\relax\else\if l#1\relax
  \rlap{\raise.5\dimen7\hbox{\kern.4ex\char'47}}#1\relax \else\if
  L#1\relax\rlap{\raise.5\dimen7\hbox{\kern.7ex
      \char'47}}#1\relax\else\message{accent \string\soft \space #1
    not defined!}#1\relax\fi\fi\fi\fi\fi\fi}

\end{document}